\documentclass[reqno]{amsart}
\usepackage{hyperref}
\usepackage{amssymb}
\usepackage{enumitem}
\newtheoremstyle{plain}{3mm}{3mm}
{\slshape}{}{\bfseries}{.}{.5em}{}
\newtheoremstyle{definition}{3mm}{3mm}
{}{}{\bfseries}{.}{.5em}{}

\theoremstyle{plain}
\newcounter{maintheoremcounter}
	
\newtheorem{Maintheorem}[maintheoremcounter]{Theorem}	
\newtheorem{Theorem}{Theorem}[section]
\newtheorem*{theorem}{Theorem}

\newtheorem{Lemma}[Theorem]{Lemma}
\newtheorem{Proposition}[Theorem]{Proposition}
\newtheorem{Corollary}[Theorem]{Corollary}
\newtheorem*{TheoremSzemeredi}{Szemer{\'e}di's theorem}

\theoremstyle{definition}
\newtheorem{Definition}[Theorem]{Definition}
\newtheorem{Remark}[Theorem]{Remark}

\newcommand{\define}[1]{{\textsl{#1}}}
\newcommand{\N}{\mathbf{N}}
\newcommand{\Z}{\mathbf{Z}}
\newcommand{\R}{\mathbf{R}}
\newcommand{\Q}{\mathbf{Q}}

\renewcommand{\epsilon}{\varepsilon}
\renewcommand{\leq}{\leqslant}
\renewcommand{\geq}{\geqslant}
\renewcommand{\v}[1]{{\vec{#1}}}

\newcommand{\LipZ}{\textit{L}_\Z}			
\newcommand{\LipR}{\textit{L}_\R}
\newcommand{\XR}[5]{X_\R(#1,#2,#3,#4,#5)}
\newcommand{\XZ}[4]{X_\Z(#1,#2,#3,#4)}
\newcommand{\constantone}{\delta}
\newcommand{\constanttwo}{c_1}
\newcommand{\constantthree}{C_1}
\newcommand{\constantfour}{C_2}
\newcommand{\constantfive}{\eta}
\newcommand{\constantsix}{\eta} 	
\newcommand{\constantseven}{c_3} 	
\newcommand{\constanteight}{c_4} 	
\newcommand{\constantnine}{c_7} 	
 	
\newcommand{\constantfunf}{c_5}
\newcommand{\constantsieben}{c_8}
\newcommand{\KZ}{{K_\Z}}
\newcommand{\KR}{{K_\R}}
\newcommand{\fcircell}{(f\circ\ell)}
\newcommand{\vell}{\v v_\ell}
\newcommand{\mell}{m_\ell}

\newcommand{\TheoremRef}[1]{Theorem~\hyperref[theorem:#1]
{\textcolor{blue}{\ref{theorem:#1}}}}
\newcommand{\CorollaryRef}[1]{Corollary~\hyperref[corollary:#1]
{\textcolor{blue}{\ref{corollary:#1}}}}
\newcommand{\LemmaRef}[1]{Lemma~\hyperref[lemma:#1]
{\textcolor{blue}{\ref{lemma:#1}}}}
\newcommand{\EquationRef}[1]{\hyperref[equation:#1]
{\textcolor{blue}{\eqref{equation:#1}}}}
\newcommand{\PropositionRef}[1]{Proposition~\hyperref[proposition:#1]
{\textcolor{blue}{\ref{proposition:#1}}}}
\newcommand{\RemarkRef}[1]{Remark~\hyperref[remark:#1]
{\textcolor{blue}{\ref{remark:#1}}}}
\newcommand{\DefinitionRef}[1]{Definition~\hyperref[definition:#1]
{\textcolor{blue}{\ref{definition:#1}}}}
\newcommand{\SectionRef}[1]{Section~\hyperref[section:#1]
{\textcolor{blue}{\ref{section:#1}}}}
\hypersetup{
  citecolor = blue,
  colorlinks,
  linkcolor = blue,
  urlcolor = blue
}

\usepackage{color}
\usepackage{soul,ulem}
\definecolor{dgreen}{RGB}{0,155,0}

\begin{document}
\title[Collinear points in discrete hypersurfaces]{Large subsets of discrete hypersurfaces in $\Z^d$ contain arbitrarily many collinear points}
\author{Joel Moreira \and Florian Karl Richter}
\address{\small Department of Mathematics, OSU, Columbus Ohio}
\email{\small moreira.6@osu.edu \and richter.109@osu.edu}
\begin{abstract}
In 1977 L.T. Ramsey showed that any sequence in $\Z^2$ with bounded gaps
contains arbitrarily many collinear points. Thereafter, in 1980, C. Pomerance
provided a density version of this result, relaxing the condition on the sequence
from having bounded gaps to having gaps bounded on average.

We give a higher dimensional generalization of these results.
Our main theorem is the following.
\begin{theorem}
Let $d\in\N$, let $f:\Z^d\to\Z^{d+1}$ be a Lipschitz map and
let $A\subset\Z^d$ have positive upper Banach density.
Then $f(A)$ contains arbitrarily many collinear points.
\end{theorem}
Note that Pomerance's theorem corresponds to the special case $d=1$.
In our proof,
we transfer the problem from a discrete to a continuous setting, allowing us to
take advantage of analytic and measure theoretic tools such as Rademacher's theorem.
\end{abstract}
\maketitle
\section{Introduction}\label{section:Introduction}

Ramsey theory deals with the problem of finding structured configurations in suitably large but possibly disordered sets.
The nature of the desired configurations can range from complete subgraphs of a graph to arithmetic progressions in $\Z$ to solutions of equations, such as $x+y=z$, in a countable commutative semigroup.
In this paper we deal with configurations consisting of finitely many collinear points in $\Z^d$.
The following theorem, which deals with this type of configurations, was obtained by L. T. Ramsey in 1977:

\begin{Theorem}[{\cite[Lemma 1]{Ramsey77}}]
\label{theorem:LTRamsey}
Let $M\in\N$ and suppose $\v u_1,\v u_2,\ldots\in\Z^2$ satisfies
\begin{equation}\label{equation:bdd-gaps}
\|\v u_{i+1}-\v u_i\|_2 \leq M\qquad\forall i\in\N.
\end{equation}
Then the sequence $\v u_1, \v u_2,\ldots$ contains arbitrarily many
collinear points. More precisely, for each $k\in\N$ there exists a set $X\subset\N$ with cardinality $|X|=k$ such that the set $\{\v u_i:i\in X\}$ is contained in a single line.
\end{Theorem}

A sequence that satisfies \eqref{equation:bdd-gaps} is said to have
\define{bounded gaps}.
The above theorem can be interpreted as an analogue of
van der Waerden's theorem on arithmetic progressions \cite{vdWaerden27}, which, in one of its many forms, states
that any sequence $u_1, u_2,\ldots \in \Z$ with bounded gaps
contains arbitrarily long arithmetic progressions.
The fact that a sequence in $\Z^2$ with bounded gaps may not contain arbitrarily long arithmetic progressions is a non-trivial result first obtained by J. Justin \cite{Justin72}.
When properly interpreted, Justin's construction gives a sequence with bounded gaps in $\Z^2$ without a five term arithmetic progression.
This construction was later improved by F. M. Dekking, who built a sequence $\v u_1,\v u_2,\ldots\in\Z^2$ with $\|\v u_{i+1}-\v u_i\|_2\leq 1$ that does not contain a four term arithmetic progression \cite{Dekking79}.

It is natural to ask whether a result similar to \TheoremRef{LTRamsey} holds in higher dimensions.
It follows, as an easy corollary, that any sequence in $\Z^d$ with bounded gaps will contain arbitrarily many points in the same $(d-1)$-dimensional hyperplane.
To see this, simply project any given sequence in $\Z^d$
onto $\Z^2$ and take the preimage under this projection of the set of collinear points guaranteed by \TheoremRef{LTRamsey}.

One could naively attempt to extend \TheoremRef{LTRamsey} by asking whether a sequence with bounded gaps in higher dimensional lattices contains arbitrarily many collinear points.
However, J. L. Gerver and L. T. Ramsey constructed a sequence in $\Z^3$ with bounded gaps (actually with gaps bounded by $1$) with no more than $5^{11}$ points contained in a single line \cite{Gerver_Ramsey79}.
This example shows that one needs to change the framework to obtain non-trivial generalizations of \TheoremRef{LTRamsey} to higher dimensions.

A sequence in $\Z^2$ with bounded gaps can be viewed as a Lipschitz function $f:\Z\rightarrow\Z^2$. Using this language, \TheoremRef{LTRamsey} asserts that the image of any such Lipschitz function contains arbitrarily many collinear points.
In order to increase the dimension of the range from $\Z^2$ to a higher dimensional space $\Z^{d+1}$ one must also increase the dimension of the domain from $\Z$ to $\Z^d$ in order to get similar qualitative results.
We will prove the following:

\begin{Maintheorem}\label{theorem:maincoloring}
Let $d\in\N$ and let $f:\Z^d\to\Z^{d+1}$ be a Lipschitz map.
Then there are arbitrarily many collinear points in the image of $f$.
More precisely for any $k\in\N$ there exists a set $X\subset\Z^d$ with $|X|=k$ such that $f(X)$ is contained in a single line.
\end{Maintheorem}

Observe that \TheoremRef{LTRamsey} can be derived from \TheoremRef{maincoloring} by setting $d=1$.
One can intuitively interpret \TheoremRef{maincoloring} as stating that any discrete hypersurface in $\Z^{d+1}$ contains arbitrarily many collinear points; this interpretation becomes rigorous if one defines a discrete hypersurface as a set quasi-isometric\footnote{A map $f:X\to Y$ between metric spaces is a quasi-isometry if there exist $C,M\geq1$ such that $1/Md(x,y)-C<d\big(f(x),f(y)\big)<Md(x,y)+C$ and for every $y\in Y$ there exists $x\in X$ such that $d\big(f(x),y\big)<C$.} to $\Z^d$.

A density version of van der Waerden's theorem, known as
Szemer{\'e}di's theorem, was obtained in
\cite{Szemeredi75}.
\begin{TheoremSzemeredi}
Let $A\subset\Z$ have positive upper Banach density, i.e.,
$$d^*(A):=\limsup_{L\to\infty}~\sup\left\{\left.\frac{|A\cap[N,N+L]|}L\ \right|N\in\Z\right\}>0.$$
Then $A$ contains arbitrarily long arithmetic progressions.
\end{TheoremSzemeredi}

It is not hard to see that Szemer{\'e}di's theorem is equivalent to the statement that
any sequence $u_1, u_2,\ldots \in \Z$ with gaps bounded
on average, i.e.,
any sequence with
$$\frac1m\sum_{i=1}^m |u_{i+1}-u_i | \leq M$$
for infinitely many $m\in\N$, contains
arbitrarily long arithmetic progressions.
In 1978, in analogy with Szemer{\'e}di's theorem,
C. Pomerance presented a proof of the following density version of
\TheoremRef{LTRamsey}:
\begin{Theorem}[\cite{Pomerance80}]
\label{theorem:Pomerance}
Let $M\in\N$ and suppose the sequence $\v u_1,\v u_2,\ldots\in\Z^2$ satisfies
\begin{equation}
\label{equation:bdd-gaps-in-average}
\frac1m\sum_{i=1}^m\|\v u_{i+1}-\v u_i\|_2 \leq M
\end{equation}
for infinitely many $m\in\N$.
Then the sequence $\v u_1, \v u_2,\ldots$ contains arbitrarily many
collinear points.
\end{Theorem}
It is a corollary of Pomerance's theorem that if a sequence $u_1,u_2,\ldots\in\Z$ has gaps bounded on average, then the sequence defined by $\v u_i=(i,u_i)\in\Z^2$, $i\in\N$, contains arbitrarily many collinear points.

It turns out that an extension of Pomerance's theorem to higher dimensions along the lines of \TheoremRef{maincoloring}
holds as well.
Our main theorem is then the common generalization of Theorems \ref{theorem:maincoloring} and \ref{theorem:Pomerance}.

\begin{Maintheorem}\label{theorem:A}
Suppose $f:\Z^d\rightarrow \Z^{d+1}$ is a Lipschitz
map and $A\subset\Z^d$ has positive upper Banach density (defined in \eqref{equation:Banach-density}).
Then given any positive integer $k$, there exists $X\subset A$
with $|X|=k$ such that $f(X)$ is contained in a line.
\end{Maintheorem}
Although not apparent at first, \TheoremRef{Pomerance} and the special case $d=1$ of \TheoremRef{A} are equivalent. We give a proof of this fact in \SectionRef{two}.
An intuitive interpretation of \TheoremRef{A} is that large subsets of discrete hypersurfaces contain arbitrarily many collinear points.

The paper is organized as follows: In \SectionRef{two} we explore some equivalent and related statements to our main theorem. In \SectionRef{outline} we outline the proof and state our main technical result, which is \LemmaRef{XZ}. In \SectionRef{ContinuousVersion} we prove \LemmaRef{XZ} by reducing it to a statement about Lipschitz functions on $\R^n$. Finally, \SectionRef{Conclusion} finishes the proof of \TheoremRef{A}.

\paragraph{\textbf{Acknowledgements}}
The authors wish to thank Vitaly Bergelson for helpful comments and remarks, as well as
the anonymous referees for their many pertinent suggestions.

\section{Equivalent Formulations and Corollaries of the Main Theorems}\label{section:two}

For the remainder of this paper we fix a dimension $d\in\N$.
For $p\in \{1,2\}$ we define
$$\|\v{x}-\v{y}\|_p^p= \sum_{i=1}^d |x_i-y_i|^p.$$
The \define{upper Banach density} $d^\ast$ of a set
$A\subset \Z^d$
is defined as:
\begin{equation}\label{equation:Banach-density}
d^\ast(A):=
\limsup_{L\to\infty}~\sup\left\{\left.\frac{\big|A\cap\prod[N_i,N_i+L]\big|}{L^d}\ \right|\ N_1,\dots,N_d\in\Z\right\}.
\end{equation}
Whenever $f$ is a function and $X$ is a subset of its domain, we denote by $f(X)$ the set $\{f(x):x\in X\}$.
For $a\in\N$ we denote by $[a]$ the set $\{1,\dots,a\}$.
For a finite set $X$ we let $|X|$ be its cardinality.
For $x\in\R$ let $\lfloor x\rfloor\in\Z$ be defined as the largest integer no bigger than $x$ and for $\v x=(x_1,\dots,x_d)\in\R^d$ let $\lfloor\v x\rfloor$
be defined as the vector
$(\lfloor x_1\rfloor,\dots,\lfloor x_d\rfloor)\in\Z^d$.
We denote by $S^d\subset\R^{d+1}$ the unit sphere, $S^d=\{\v x\in\R^{d+1}:\|x\|_2=1\}$.

Given $d,h\in\N$, $M>0$ and a set $Z\subset\R^d$, a function $f:Z\to\R^h$ is \define{Lipschitz} with \define{Lipschitz constant $M$} if $\|f(\v x)-f(\v y)\|_2\leq M\|\v x-\v y\|_2$ for all $\v x,\v y\in Z$.

First let us formulate a seemingly more general, but, in fact, equivalent version of \TheoremRef{A}.
\begin{Theorem}\label{theorem:B}
Let $d,h\in\N$.
Suppose $f:\Z^d\rightarrow \Z^{d+h}$ is a Lipschitz
map and $A\subset\Z^d$ has positive upper Banach density.
Then given any positive integer $k$, there exists $X\subset A$
with cardinality $|X|=k$ such that $f(X)$ is contained in a $h$-dimensional
hyperplane of $\Z^{d+h}$.
\end{Theorem}
When $h=1$, \TheoremRef{B} reduces to \TheoremRef{A}.
To deduce \TheoremRef{B} from \TheoremRef{A}, compose $f$ with the projection $\pi:\Z^{d+h}\to\Z^{d+1}$, find a line in $\Z^{d+1}$ which contains $\pi(f(X))$ and notice that the pre-image of a line under $\pi$ is an $h$-dimensional affine subspace in $\Z^{d+h}$.

As is usual with Ramsey theory results, there is an equivalent formulation of \TheoremRef{A} in finitistic terms:
\begin{Theorem}\label{theorem:mainfinitistic}
  Let $d,k\in\N$ and let $\delta,M>0$.
  There exists $L=L(d,k,\delta,M)\in\N$ such that for any Lipschitz function $f:\Z^d\to\Z^{d+1}$ with Lipschitz constant $M$ and any $A\subset[L]^d$ with cardinality $|A|>\delta L^d$ one can find a subset $X\subset A$ with $|X|=k$ such that $f(X)$ is contained in a line.
\end{Theorem}

\begin{Proposition}\label{proposition:finitistic}
  \TheoremRef{mainfinitistic} and \TheoremRef{A} are equivalent.
\end{Proposition}

\begin{proof}It is trivial to see that \TheoremRef{mainfinitistic} implies \TheoremRef{A}.
To prove the converse suppose, for the sake of a contradiction, that \TheoremRef{mainfinitistic} is false.
Thus there are $d,k,\delta,M$ such that for every $L\in\N$ one can find a set $A_L\subset[L]^d$ with $|A_L|>\delta L^d$ and a Lipschitz function $f_L:[L]^d\rightarrow\Z^{d+1}$ with Lipschitz constant $M$ such that for any $X\subset A_L$ with $|X|=k$, the image $f_L(X)$ is not contained in a single line.

Let the sequence $(\v N_L)_{L=1}^\infty$ in $\Z^{d+1}$ be defined recursively by letting $\v N_1=\v 0$ and, for each $L>1$,
by choosing $\v N_L\in\Z^{d+1}$ such that $\v N_L+f_L([L]^d)$ is disjoint from all the lines which contain at least two points of the union $\bigcup_{j=1}^{L-1}\v N_j+f_L([j]^d)$.

Next let $(\v M_L)_{L=1}^\infty$ be a sequence in $\Z^d$ such that the $\|.\|_1$-distance between $\v M_L+[L]^d$ and the union $\bigcup_{j=1}^{L-1}\v M_j+[j]^d$ is at least $\|\v N_L\|_1$.
Let $A$ be the union of $\v M_L+A_L$ over all $L\in\N$; it is clear that $d^\ast(A)\geq\delta$.

Finally, define $g:\Z^d\to\Z^{d+1}$ to be a Lipschitz function with Lipschitz constant $M$ such that for every $L\in\N$ and $\v x\in\v M_L+[L]^d$ we have $g(\v x)=\v N_L+f_L(\v x-\v M_L)$.

According to \TheoremRef{A} one can find $X\subset A$ with $|X|=k+1$ such that $g(X)$ is contained in a line.
Find the maximal $L\in\N$ for which there exists some $\v x\in X$ with $\v x\in\v M_L+[L]^d$.
Observe that $g(\v x)\in\v N_L+f_L([L]^d)$, so the line which contains $g(X)$ cannot contain more than one point from $\bigcup_{j=1}^{L-1}\v N_j+f_L([j]^d)$.
Therefore, there is a subset $Y\subset X$ with $|Y|=k$ such that $Y\subset\v M_L+[L]^d$.
Since $g(Y)$ is still contained in a single line and $g(Y)=\v N_L+f_L(Y-\v M_L)$, the set $\tilde Y=Y-\v M_L$ is a subset of $A_L$, with cardinality $k$ such that $f(\tilde Y)$ is contained in a line, thus contradicting the construction.
This contradiction finishes the proof.
\end{proof}
Pomerance's original formulation of \TheoremRef{Pomerance} in \cite{Pomerance80} was in finitistic terms.
More precisely, he showed that for every $k,M\in\N$ there exists $n=n(k,M)\in\N$ such that whenever $\v u_0,\ldots,\v u_n\in\Z^2$ satisfy $\sum_{i=1}^n\|\v u_i-\v u_{i-1}\|_2\leq nM$, there are $k$ collinear points among $\v u_1,\ldots,\v u_n$.
This statement clearly implies \TheoremRef{Pomerance}; the reverse implication can be deduced similarly to the proof of \PropositionRef{finitistic}.

As mentioned in the Introduction, the case $d=1$ of \TheoremRef{A} is equivalent to \TheoremRef{Pomerance}.
To see how Pomerance's theorem implies the case $d=1$ of \TheoremRef{A}, we will use the finitistic versions of both theorems.

Let $k,\delta,M$ be as in \TheoremRef{mainfinitistic} and let $f:\Z\to\Z^{2}$ be a Lipschitz function with constant $M$.
Let $L=M/\delta$ and let $n=n(k,L)$ be given by Pomerance's theorem.
Finally let $N\geq n/\delta$.

Take any $A\subset[N]$ with $|A|>\delta N\geq n$ and order it,
$A=\{a_1<\cdots<a_n\}$.
Let $\v u_i=f(a_i)$.
It now suffices to show that the average gap of the sequence $\v u_1,\ldots,\v u_n$ is at most $L$ and the result will follow by Pomerance's theorem.
Indeed we have
$$\sum_{i=1}^{n-1}\|\v u_{i+1}-\v u_i\|_2=\sum_{i=1}^{n-1}\|f(a_{i+1})-f(a_i)\|_2\leq M\sum_{i=1}^{n-1}a_{i+1}-a_i=M(a_n-a_1)\leq nL.$$

To prove the converse direction (i.e., that \TheoremRef{A} with $d=1$ implies \TheoremRef{Pomerance}), let $\v u_1,\v u_2,\ldots$ be a sequence in $\Z^2$ with gaps bounded on average by $M$.
For each consecutive pair
$\v u_i,\v u_{i+1}$ consider a path of minimal $\|\cdot\|_1$ length connecting
$\v u_i$ with $\v u_{i+1}$.
Each such path will have length
$\|\v u_{i+1}-\v u_i\|_1$ and stringing them together defines a
Lipschitz function $f:\Z\rightarrow \Z^2$.
Next construct the set
$A=\{a_i\}_{i\in\N}$ recursively by setting $a_1=1$ and
$a_{i+1} = a_i + \|\v u_{i+1}-\v u_i\|_1$.
It is then easy to check that $A$ has density bounded
from below by $1/M$ and that $f(a_i)=\v u_i$. Thus by applying
\TheoremRef{A} we can find $X\subset A$ with $|X|=k$ such that
$f(X)\subset\{\v u_1, \v u_2,\ldots\}$ is collinear.

As a Corollary of \TheoremRef{B} we immediately obtain the
following ``coloring'' version  of our main theorem:

\begin{Corollary}\label{Corollary:B}
Let $n,h,M\in\N$, let $f:\Z^{n}\rightarrow \Z^{n+h}$ be a Lipschitz
map and suppose $\Z^{n+h}$ has been colored with finitely many colors.
Then given any positive integer $k$, there exists a subset
$X\subset \Z^n$ of size $k$ with
$f(X)$ monochromatic and contained in a
$n$-dimensional subspace of $\Z^{n+h}$.
\end{Corollary}
Similarly to \PropositionRef{finitistic}, the case $h=1$ of this corollary is equivalent to \TheoremRef{maincoloring}.
\section{Outlining the proof of \TheoremRef{A}}
\label{section:outline}

Throughout the rest of this paper, let $d,k\in\N$, $M\in\R^+$ and
$A\subset\Z^d$ with $d^\ast(A)>0$ be arbitrary but fixed.
In the following, these four
parameters will be invisible
in the notation to reduce the amount of subscripts.

Let $\LipZ$ denote the set of all Lipschitz functions
$f: \Z^d \rightarrow \Z^{d+1}$ with Lipschitz constant $M$, and with the property that there
exists no set $X\subset A$ with $|X|\geq k$ such that $f(X)$ is
collinear. Thus \TheoremRef{A} is proven if we can show that
$\LipZ$ is in fact the empty set for all $d,k,M,A$.

\begin{Definition}
A \define{generalized line segment} is a function
$\ell: [0,1]\rightarrow \Z^d$ of the form
\[
\ell(t)=\lfloor(1-t)\v x+ t\v y \rfloor
\]
for some $\v x,\v y\in\R^d$.
Given a generalized line segment $\ell$ we denote by
\define{$\mell$} the distance $\mell=\|\ell(1)-\ell(0)\|_2$.
\end{Definition}

The underlying argument of the proof goes back to Ramsey's paper \cite{Ramsey77}, and was adapted by Pomerance in \cite{Pomerance80}.
The basic idea is to find a long, narrow cylinder in $\Z^{d+1}$ which contains ``many'' points from $f(A)$.
We can then cover this cylinder with not too many lines that are almost parallel to the axis of the cylinder, which allows us to find some line containing at least $k$ points.
However, our methods to find such a cylinder differ significantly from both Ramsey and Pomerance, mainly due to our appeal to the classical Rademacher's theorem:

 \begin{Theorem}[Rademacher's Theorem, cf. {\cite[Theorem 3.1]{Heinonen05}}]
   Let $d,h\in\N$ be arbitrary dimensions, let $U\subset\R^d$ and let $f:U\to\R^h$ be Lipschitz.
   Then $f$ is differentiable at (Lebesgue) almost every point $\v x\in U$.
 \end{Theorem}

Rademacher's theorem tells us that a Lipschitz function is almost everywhere locally `flat' in a certain sense.
We will use this property to find the cylinder with the desired properties.

\begin{Definition}\label{definition:XZ}
Assume $\LipZ$ is non-empty. Given $f\in\LipZ$, $\epsilon,\constantone>0$ and
$\v w=(w_1,\dots,w_{d+1})\in S^d\subset\R^{d+1}$
we define \define{$\XZ{f}{\epsilon}{\constantone}{\v w}$} to be the
collection of all
generalized line segments $\ell:[0,1]\to\Z^d$ with $\epsilon m_\ell>14\sqrt{d}$ and satisfying the following properties:

\begin{enumerate}
[label=(z-\roman{enumi})~~~~,ref=(z-\roman{enumi}),leftmargin=*]
\item
\label{item:z-i}
$\left\|\vell-\v w\right\|_2<\epsilon,$
where $\vell$ denotes the `mean slope' of $\fcircell$,
$$
\vell=
\frac{\fcircell(1)-\fcircell(0)}
{\big\|\fcircell(1)-\fcircell(0)\big\|_2}.
$$

\item
\label{item:z-ii}
For every $t\in [0,1]$,
$$\left\|\fcircell(t)-\big[(1-t)\fcircell(0)+
t\fcircell(1)\big]\right\|_2<\epsilon Mm_\ell.$$
Roughly speaking, this condition states that the image of the generalized line segment
$\ell$ under $f$ remains relatively close to a line.
\item
\label{item:z-iii}
If we let $\KZ=\KZ(\epsilon,\ell)$ be the cylinder defined by
$$
\KZ=\big\{\v z \in \Z^d:\min_{t\in [0,1]}\|\v z-
\ell(t)\|_2\leq\epsilon\mell\big\}
$$ then
$\big|A\cap\KZ\big|>\constantone|\KZ|.$
\end{enumerate}
\end{Definition}

\begin{Lemma}\label{lemma:XZ}
Suppose $\LipZ$ is non-empty. Then for every $f\in\LipZ$ there exists $\delta>0$ and $\v w\in S^d$ such that the set
$\XZ{f}{\epsilon}{\constantone}{\v w}$ is non-empty for all sufficiently small $\epsilon>0$.
\end{Lemma}

It is our goal to use Rademacher's theorem to deduce \LemmaRef{XZ}.
In order do this, we need first to convert \LemmaRef{XZ} into a continuous version; this is done by \TheoremRef{continuous} in \SectionRef{ContinuousVersion}.
In \SectionRef{Conclusion} we use \LemmaRef{XZ} together with the methods developed by Ramsey and Pomerance to finish the proof.

\section{Deducing \LemmaRef{XZ} from a Continuous Version}
\label{section:ContinuousVersion}

We use $\lambda$ to represent the Lebesgue measure on $\R^d$ and
define the ball $B_\R(\v x,r)=\{\v y\in\R^d : \|\v x-\v y\|_2\leq r\}$ for any $\v x\in\R^d$ and $r>0$.
\begin{Definition}\label{definition:XR}
Let $T:[-1,1]^d\to[-1,1]^{d+1}$ be a Lipschitz function with Lipschitz constant $1$,
let $\phi: [-1,1]^d\rightarrow [0,1]$ be Lebesgue measurable,
let $\v x\in[-1,1]^d$ and
let $\v w\in S^d\subset\R^{d+1}$.
For each $\epsilon,\delta>0$ we define the set
\define{$\XR{T}{\epsilon}{\delta}{\v w}{\v x}$} as the set of all
$\v y\in[-1,1]^d$ with the following properties:
\begin{enumerate}
[label=(r-\roman*)~~~~,ref=(r-\roman*)]
\item
$$
\left\|\frac{T(\v y)-T(\v x)}{\big\|T(\v y)-T(\v x)\big\|_2}-
\v w\right\|_2<\epsilon.
$$
This asserts that the direction of the line segment
connecting $T(\v x)$ and $T(\v y)$ is approximately equal to $\v w$.
\item
For every $t\in[0,1]$,
$$
\Big\|T\big((1-t)\v x+t\v y\big)-
\big[(1-t)T(\v x)+tT(\v y)\big]\Big\|_2<\epsilon\|y-x\|_2.
$$
Similar to condition \ref{item:z-ii}, this condition states that the image under $T$ of
the line segment connecting $\v x$ and $\v y$
remains relatively close to a line.
\item
If we let
$$
\KR=\KR(\epsilon, \v x, \v y)=
\{(1-t)\v x+t\v y: t\in[0,1]\}+B_\R(\v 0,\epsilon\|\v y-\v x\|_2)
$$ then
$$
\frac{1}{\lambda(\KR)}\int_\KR\phi~d\lambda>\constantone.
$$
In other words, $\phi$ gives enough mass to a thin cylinder around the
segment connecting $\v x$ and $\v y$.
\end{enumerate}
\end{Definition}

We denote by $\LipR$ the set of Lipschitz functions
$T:[-1,1]^d\to[-1,1]^{d+1}$
with Lipschitz constant $1$ and
with the property that for any
point $\v x\in(-1,1)^d$ where $T$ is differentiable, the
derivative is nonzero (i.e., some partial derivative is nonzero).

\begin{Theorem}\label{theorem:continuous}
Let $T\in\LipR$ and let $\phi\in L^\infty\big([-1,1]^d\big)$ be
non-negative with $\int\phi d\lambda>0$.
Then there exist $\constantone>0$, $\v x\in[-1,1]^d$ and
$\v w\in S^d$ such that for every sufficiently small $\epsilon>0$ the set
$\XR{T}{\epsilon}{\constantone}{\v w}{\v x}$ is non-empty.
\end{Theorem}
In order to prove \TheoremRef{continuous} we will need
the following Lemma.

\begin{Lemma}\label{lemma:partition}
Let $R=\{r_n:n\in\N\}$ be an infinite subset of $\R^+$, let
$\phi\in L^\infty\big([-1,1]^d\big)$ be a non-negative function and let $\constantsix>0$. Assume that
\begin{equation}\label{equation:lemma_partition_condition}
\frac{1}{\lambda(B_\R(\v 0,r))}
\int_{B_\R(\v 0,r)}\phi~d\lambda\geq\constantsix\qquad\forall r\in R.\end{equation}
Let $\mu$ denote the $(d-1)$-dimensional Hausdorff measure
on $\R^d$.
Then there exists some constant $\constanttwo>0$ that only depends on
the dimension $d$ and a set $P\subset S^{d-1}$ with $\mu(P)>0$ and such that for any $\v z\in P$ there exists an infinite
subset $R'(\v z)\subset R$ such that
\[
\liminf_{\epsilon\rightarrow 0}\frac1
{\lambda(r\KR)}\int_{r\KR}
\phi~d\lambda\geq \constanttwo\constantsix ,\qquad\forall r\in R'(\v z)
\]
where
$\KR=\KR(\epsilon,\v 0, \v z)$ is as in \DefinitionRef{XR}.
\end{Lemma}

\begin{proof}

Without loss of generality we assume that $\|\phi\|_\infty\leq1$.
For each $r>0$, the measure space $(B_\R(\v0,r),\lambda)$ can be decomposed as the product of the measure spaces $(S^{d-1},\mu)$ and $([0,r],t^{d-1}dt)$.
Letting
\[
\psi_r(\v z)= \int_{[0,r]}
\phi(t\v z) t^{d-1}dt,
\]
we deduce from \EquationRef{lemma_partition_condition} that for every $r\in R$ we have
\[
\constantsix\lambda(B_\R(\v 0,r))\leq\int_{S^{d-1}} \psi_r(\v z) d\mu(\v z).
\]

Let $A_r$ be the set of those
$\vec z\in S^{d-1}$ for which
\begin{equation}\label{eq_psi}
\psi_r(\v z)> \frac{\constantsix\lambda(B_\R(\v 0,r))}{2\mu(S^{d-1})}.
\end{equation}
Observe that $|\psi_r(\v z)|\leq r^d/d$, and hence
$$\constantsix\lambda(B_\R(\v 0,r))\leq \frac{r^d}d\mu(A_r)+\frac{\constantsix\lambda(B_\R(\v 0,r))}{2\mu(S^{d-1})}\mu(S^{d-1}).$$
It follows that  $\mu(A_r)\geq c_0\constantsix/2$ for some constant $c_0$ which only depends on the dimension $d$.

Next we apply Lebesgue's differentiation theorem to
find a set $B_r\subset A_r$ with
$\mu(A_r)=\mu(B_r)$ and such that for every
$\vec z\in B_r$
\[
\psi_r(\v z)=\lim_{\epsilon\rightarrow 0}
\frac{1}{\mu(D_\epsilon(\v z))}
\int_{D_\epsilon(\v z)} \psi_r d\mu
\]
where $D_\epsilon(\v z)= B_\R(\v z,\epsilon)\cap S^{d-1}$.
It follows from (reverse) Fatou's lemma that the set
$$P=\limsup B_{r_m} =\bigcap_{n=1}^\infty\bigcup_{m=n}^\infty B_{r_m}$$
has measure $\mu(P)\geq c_0\constantsix/2>0$.
By construction, for every $\v z\in P$
there exists an infinite subset $R'\subset R$ such that
\[
\v z\in\bigcap_{r\in R'} B_r.
\]
This implies that for every $r\in R'$ and sufficiently small
$\epsilon$ we
can assume that
\[
\int_{D_\epsilon(\v z)}\psi_r~d\mu\geq
\frac{\constantsix\lambda(B_\R(\v 0,r))}{3\mu(S^{d-1})}\mu(D_\epsilon(\v z)).
\]
Let $C_r=\{t \v u : t\in[0,r],~\v u\in D_\epsilon(\v z)\}$. Then for $r\in R'$
\begin{eqnarray*}
\int_{C_r}\phi~d\lambda
&=&
\int_{D_\epsilon(\v z)}
\int_{[0,r]}\phi(t\v u)t^{d-1}~dt~d\mu(\v u)
\\
&=&
\int_{D_\epsilon(\v z)} \psi_r(\v u) d\mu(\v u)
\\
&\geq &
\frac{\constantsix\lambda(B_\R(\v 0,r))}{3{\mu(S^{d-1})}}\mu(D_\epsilon(\v z)).
\end{eqnarray*}

Finally we note that the cylinder
$r\KR=\KR(\epsilon, \v 0, r\v z)$
contains the cone $C_r$ and that for fixed $d$ the quotient
\[
\frac{\lambda(B_\R(\v 0,r))\mu(D_\epsilon(\v z))}
{3\lambda(r\KR)\mu(S^{d-1})}
\]
is constant. From this the lemma follows.
\end{proof}

\begin{Remark}
Observe that condition (\ref{equation:lemma_partition_condition}) in \LemmaRef{partition} can be replaced with
$$\frac{1}{\lambda(B_\R(\v x,r))}
\int_{B_\R(\v x,r)}\phi~d\lambda\geq
\constantsix$$
for an arbitrary point $\v x\in[-1,1]^d$.
In this case, the cylinder $\KR$ in the conclusion becomes $\KR=\KR(\epsilon,\v x, \v z)$.

To see this one can apply \LemmaRef{partition} to the function $\tilde\phi(\v y)=\phi(\v y-\v x)$.
\end{Remark}

\begin{proof}[Proof of \TheoremRef{continuous}]
First let us invoke Lebesgue's differentiation theorem as well as
Rademacher's Theorem to find a set $X\subset [-1,1]^d$ with full Lebesgue measure
such that for every $\v x\in X$ the map $T$ is differentiable
at $\v x$ and
\[
\phi(\v x)=\lim_{r\rightarrow 0} \frac{1}{\lambda(B_\R(\v x,r))}
\int_{B_\R(\v x,r)} \phi~d\lambda.
\]
Pick any point $\v x\in X$ such that $\phi(\v x)>0$. Then,
since $T$ is differentiable at $\v x$, there exists a linear map
$J:\R^d\to\R^{d+1}$, the Jacobian of $T$ at $\v x$, such that
$T(\v y)$ can be written as
\begin{equation}\label{equation:continuous}
T(\v y)=T(\v x)+ J\cdot(\v y - \v x)+ e(\v y-\v x) \|\v y-\v x\|_2,
\end{equation}
where the error term $e(\v z)$ is continuous and satisfies
$e(\v 0)=0$.
Since $T\in\LipR$, $J\neq0$.
Next take $R=\{\frac1n\}_{n\geq n_0}$.
For $n_0$ large enough we can apply \LemmaRef{partition} to $R$, $\phi$ and $\v x$.
Let $P\subset S^{d-1}$ be the set obtained this way.
Since $P$ has positive measure, it spans $\R^d$, and because
$J$ is a non-zero linear map, there exists some $\v z\in P$
for which $J\cdot\v z\neq0$. Since $\v z \in P$ we can find an infinite set
$R'\subset R$ such that
\[
\liminf_{\epsilon\rightarrow 0} \frac{1}{\lambda(r\KR)}
\int_{r\KR} \phi~d\lambda\geq \constanttwo \phi(\v x),
\qquad\forall r\in R'.
\]
where $\KR=\KR(\epsilon,\v x, \v z)$ is as in \DefinitionRef{XR}.

Set $\constantone= \frac{\constanttwo \phi(x)}{2}$ and set
$\v w = \tfrac{J \v z}{\|J\v z\|_2}\in S^d$.
We claim that with this choice of $\constantone$ and $\v w$
the set
$\XR{T}{\epsilon}{\constantone}{\v w}{\v x}$ is non-empty
for all sufficiently small $\epsilon>0$. To show this,
take $r\in R'$ sufficiently
small such that
$e(\v u)< \min(\epsilon/2,\|J\cdot z\|_2\epsilon/3)$ for all $\v u$ with
$\|\v u\|_2\leq r$. Thereafter set
$\v y = \v x + r \v z$. It follows from
\EquationRef{continuous} that
$$\left\|\frac{T(\v y)-T(\v x)}{\big\|T(\v y)-T(\v x)\big\|_2}-
\v w\right\|_2<\epsilon.$$
Also, provided that $\epsilon$ was chosen sufficiently small,
we have
\[
\int_{\KR(\epsilon,\v x,\v y)} \phi~d\lambda\geq \constantone \lambda\big(\KR(\epsilon,\v x,\v y)\big).
\]
At last, note that the distance between
$T\big(t\v x + (1-t)\v y\big)$ and
$t T\big(\v x\big) + (1-t)T\big(\v y\big)$ is equal to
$(1-t)r$ times the distance between
$e\big((1-t)r(\v z)\big)$ and
$e\big(r\v z\big)$, which indeed
is smaller than $\epsilon \|\v y-\v x\|_2=\epsilon r$.
\end{proof}

The rest of this section is dedicated to deriving
\LemmaRef{XZ} from \TheoremRef{continuous}.
Assume $\LipZ$ is non-empty and let $f\in\LipZ$.
Recall that every $f$ in $\LipZ$ has Lipschitz constant $M$.
By definition (see \EquationRef{Banach-density})
one can find a sequence $(\v z_r)_{r\in\N}$ in $\Z^d$
such that
\begin{equation}\label{equation:density}\limsup_{r\to\infty}
\frac{\Big|A\cap\big([-r,r)^d+\v z_r\big)\Big|}{(2r)^d}=
d^*(A).\end{equation}
One can rarefy the sequence $\big(r\big)_{r\in\N}$, to say $(r(i))_{i\in\N}$, so that the $\limsup$ in \EquationRef{density} is replaced by $\lim$.
For each $i$ let $V_i:[-1,1]^d\to[-1,1]^{d+1}$ be the map

\begin{equation}
\label{equation:Ref8}
V_i(\v x)=\frac1{Mr(i)}
\Big[f\big(\lfloor r(i)\v x\rfloor+\v z_{r(i)}\big)-f(\v z_{r(i)})\Big].
\end{equation}
One can further rarefy the sequence $(r(i))_{i\in\N}$, so that
$$T(\v x):=\lim_{i\to\infty}V_i(\v x)$$
exists for every $\v x\in [-1,1]^d\cap\Q^d$.
One can easily deduce that for any $\v x,\v y\in [-1,1]^d\cap\Q^d$ we have
\begin{equation*}
\big\|T(\v x)-T(\v y)\big\|_2\leq \|\v x-\v y\|_2.
\end{equation*}

This implies that $T$ can be extended to
$[-1,1]^d$ as a Lipschitz function with Lipschitz constant $1$.
Since $\Q^d$ is dense in $\R^d$
(and Lipschitz functions are continuous),
this extension is unique.
\begin{Lemma}
\label{lemma:Tuniformconvergence}
$V_i\to T$ uniformly on $[-1,1]^d$.
\end{Lemma}
\begin{proof}
Fix $\epsilon>0$.
One can find a finite set $F\subset[-1,1]^d$
such that any $\v x\in[-1,1]^d$ satisfies
$\|\v x-\v y\|_2<\epsilon/4$ for some
$\v y=\v y(\v x)\in F$.
Let $i\in\N$ be large enough so that
$\|V_j(\v y)-T(\v y)\|_2<\epsilon/4$
for all $j\geq i$ and $\v y\in F$,
and such that $d/r(i)<\epsilon/4$.

Let $\v x\in[-1,1]^d$ be arbitrary,
let $\v y\in F$ be such that
$\|\v x-\v y\|_2<\epsilon/4M$ and let $j\geq i$.
Then
\begin{eqnarray*}\|T(\v x)-V_j(\v x)\|_2
&\leq&
\|T(\v x)-T(\v y)\|_2+\|T(\v y)-V_j(\v y)\|_2+\|V_j(\v y)-V_j(\v x)\|_2
\\&\leq&
\|\v x-\v y\|_2+\frac\epsilon4+\|\v x-\v y\|_2+\frac d{r(j)}\\
&\leq&\epsilon.
\end{eqnarray*}
\end{proof}

Next consider the sequence
$(\phi_i)_{i\in\N}$ in $L^2\big([-1,1]^d\big)$
defined by
$$
\phi_i(\v x)=
1_A\big(\lfloor r(i)\v x\rfloor+\v z_{r(i)}\big).
$$
Observe that
$$
\int\phi_id\lambda=
\frac{\bigg|A\cap\Big(\big[-r(i),r(i)\big)^d+\v z_{r(i)}\Big)\bigg|}{r(i)^d}
$$
and hence, using \EquationRef{density},
$$
\lim_{i\to\infty}\int\phi_id\lambda=
2^d d^\ast(A).
$$
Rarifying $\big(r(i)\big)_{i\in\N}$ further, if necessary,
we can assume that $\phi=\lim\phi_i$ exists
in the weak topology of $L^2$.
Then we have
$$
\int\phi d\lambda=2^d d^\ast (A).
$$
Observe that, since $\langle\phi_i,1_B\rangle\leq\mu(B)$
where $B=\{x:\phi(x)>1+\epsilon\}$ it follows that $\phi$
takes values in $[0,1]$.

The goal is to derive \LemmaRef{XZ} for
$f\in\LipZ$ (assuming $\LipZ\neq \emptyset$)
by applying \TheoremRef{continuous}
to $T\in\LipR$.
Before we can do this, we need to check that $T\in\LipR$.

\begin{Lemma}
Assume $\LipZ$ is non-empty and let $f\in\LipZ$. Let $T$ be defined by the construction above.
If $T$ is differentiable at a point
$\v x\in[-1,1]^d$, then the derivative is nonzero.
\end{Lemma}

\begin{proof}
Assume, for the sake of a contradiction,
that $T$ is differentiable at a point
$\v x\in[-1,1]^d$ and that the derivative is $0$.
Let $\epsilon>0$ to be determined later and
find $\delta>0$ such that whenever
$\|\v y-\v x\|_2<\delta$ we have
$\|T(\v y)-T(\v x)\|_2<\epsilon\|\v y-\v x\|_2
\leq\epsilon\delta$.
Choose $i$ large enough so that
$\|V_i(\v y)-T(\v y)\|_2<\epsilon\delta$
for any $\v y\in[-1,1]^d$ and let
$$
\v u:=Mr(i)V_i(\v x)+f(\v z_{r(i)})=
f\big(\lfloor r(i)\v x\rfloor+\v z_{r(i)}\big)\in\Z^{d+1}
$$
Now take
$\v v\in\v z_{r(i)}+r(i)B_\R(\v x,\delta)\cap\Z^d$
and let $\v y=(\v v-\v z_{r(i)})/r(i)$.
We have
\begin{eqnarray*}
\|f(\v v)-\v u\|_2&=&Mr(i)\|V_i(\v y)-V_i(\v x)\|_2\\
&\leq&
Mr(i)\Big(\|T(\v y)-T(\v x)\|_2+\|V_i(\v y)-T(\v y)\|_2+\|T(x)-V_i(\v x)\|_2\Big)\\
&\leq&
Mr(i)(\epsilon\delta+2\epsilon\delta)=3Mr(i)\epsilon\delta.
\end{eqnarray*}
We just showed that
$f\big(\v z_{r(i)}+r(i)B_\R(\v x,\delta)\cap\Z^d\big)
\subset B_\Z\big(\v u,3Mr(i)\epsilon\delta\big)$.
On the one hand,
$$
\big|\v z_{r(i)}+r(i)B_\R(\v x,\delta)\cap\Z^d\big|
\geq \constantthree \big(\delta r(i)\big)^d
$$
for some $\constantthree>0$ that only
depends on the dimension $d$.
On the other hand, the ball
$B_\Z\big(\v u,3Mr(i)\epsilon\delta\big)$ can be
covered with no more than
$\constantfour\big(3Mr(i)\epsilon\delta\big)^d$
vertical lines, for some other constant
$\constantfour>0$ that only depends on the dimension $d$.
Since each line in $\Z^{d+1}$ contains the image
(under $f$) of at most $k$ points, we deduce that
$$
\constantthree\big(\delta r(i)\big)^d
\leq k\constantfour\big(3Mr(i)\epsilon\delta\big)^d.
$$
Rearranging, we get
$\epsilon\geq\sqrt[d]{C_1/(3^dM^dkC_2)}$, so, by choosing $\epsilon$
small enough (depending only on $k$, $M$ and $d$),
we obtain the desired contradiction.\end{proof}

Now, we can apply \TheoremRef{continuous} to $T=\lim V_i$
and $\phi=\lim\phi_i$ in order to find
$\constantone'>0$, $\v x\in[-1,1]^d$ and $\v w\in S^d$ such
that for every sufficiently small $\epsilon'>0$ the set
$\XR{T}{\epsilon'}{\constantone'}{\v w}{\v x}$ is nonempty.
We will show that $\XZ{f}{\epsilon}{\constantone}{\v w}$
is non-empty for $\constantone:=\frac{\constantone'}{4^d}$ and
$\epsilon:=\epsilon'/2$. To prove this claim, take any $\v y\in\XR{T}{\epsilon'}{\constantone'}{\v w}{\v x}$ and
put $\constantfive:=\epsilon\|\v x-\v y\|_2/16$.
In view of \LemmaRef{Tuniformconvergence} we can find
$i\in\N$ large enough so that the following three conditions are satisfied:
\begin{enumerate}
[label=(i-\arabic*)~~~~,ref=(i-\arabic*)]
  \item\label{item:i1}
  $\left\|\frac{V_i(\v x)-V_i(\v y)}{\left\|V_i(\v x)-V_i(\v y)\right\|_2}-\frac{T(\v x)-T(\v y)}{\left\|T(\v x)-T(\v y)\right\|_2}\right\|_2<\frac\epsilon2$;
  \item\label{item:i2}
  $\|V_i(\v z)-T(\v z)\|_2<\constantfive, \qquad\forall\v z\in[-1,1]^d$;
  \item\label{item:i3}
  $r(i)>\frac{\sqrt{d}}{\eta}$.
\end{enumerate}
Let $\ell$ be the generalized line segment defined by
$\ell(0)=\lfloor r(i)\v x\rfloor+\v z_{r(i)}$ and
$\ell(1)=\lfloor r(i)\v y\rfloor+\v z_{r(i)}$.
The proof of \LemmaRef{XZ} will be completed with the following lemma.
\begin{Lemma}
$\ell\in\XZ{f}{\epsilon}{\constantone}{\v w}$.
\end{Lemma}

\begin{proof}
First notice that $m_\ell=\big\|\lfloor r(i)\v x\rfloor-\lfloor r(i)\v y\rfloor\big\|_2\geq r(i)\|\v x-\v y\|_2-2\sqrt{d}$, and hence by condition \ref{item:i3} we have $\epsilon m_\ell>16\sqrt{d}-2\epsilon\sqrt{d}\geq14\sqrt{d}$.
By equation \EquationRef{Ref8} we have $\fcircell(0)=r(i)MV_i(\v x)+f(\v z_{r(i)})$ and
$\fcircell(1)=r(i)MV_i(\v y)+f(\v z_{r(i)})$.
Therefore,
\begin{eqnarray*}
\left\|\frac{\fcircell(1)-\fcircell(0)}{\Big\|\fcircell(1)-\fcircell(0)\Big\|_2}-\v w\right\|_2
&=&
\left\|\frac{r(i)MV_i(\v x)-r(i)MV_i(\v y)}{\Big\|r(i)MV_i(\v x)-r(i)MV_i(\v y)\Big\|_2}-\v w\right\|_2\\
&=&
\left\|\frac{V_i(\v x)-V_i(\v y)}{\Big\|V_i(\v x)-V_i(\v y)\Big\|_2}-\v w\right\|_2\\
&\leq&
\frac\epsilon2+\left\|\frac{T(\v x)-T(\v y)}{\Big\|T(\v x)-T(\v y)\Big\|_2}-\v w\right\|_2\\
&\leq&\epsilon,
\end{eqnarray*}
where the first inequality follows from \ref{item:i1}.
This proves \ref{item:z-i}. For the second condition, \ref{item:z-ii},
let $t\in[0,1]$ and observe that
on the one hand
$$\ell(t)=\lfloor(1-t)\ell(0)+t\ell(1)\rfloor=
\big\lfloor(1-t)\lfloor r(i)\v x\rfloor+t\lfloor r(i)\v y\rfloor\big\rfloor+\v z_{r(i)}$$
which implies that
\begin{eqnarray*}
\fcircell(t)&=&r(i)MV_i\left(\frac{(1-t)\lfloor r(i)\v x\rfloor+t\lfloor r(i)\v y\rfloor}{r(i)}\right)
+f(\v z_{r(i)})\\&=&r(i)MV_i\big((1-t)\v x+t\v y\big)+f(\v z_{r(i)})+\v e
\end{eqnarray*}
for some $\v e$ with $\|\v e\|_\infty\leq2M\sqrt{d}$.
On the other hand,
$$(1-t)\fcircell(0)+t\fcircell(1)=(1-t)r(i)MV_i(\v x)+tr(i)MV_i(\v y)+f(\v z_{r(i)})$$
By combining both (and by using \ref{item:i2} and \ref{item:i3})
we get
\begin{eqnarray*}
&&\left\|\fcircell(t)-\Big[(1-t)\fcircell(0)+t\fcircell(1)\Big]\right\|_2\\
&=&
r(i)M\left\|V_i\big((1-t)\v x+t\v y\big)-\big[(1-t)V_i(\v x)+
tV_i(\v y)\big]+\frac{\v e}{r(i)M}\right\|_2\\
&\leq&
r(i)M\left(\Big\|T\big((1-t)\v x+t\v y\big)-\big[(1-t)T(\v x)+tT(\v y)\big]\Big\|_2+2\eta+\frac{2\sqrt{d}}{r(i)}\right)
\\&\leq&r(i)M\big(\epsilon'\|\v x-\v y\|_2+2\eta+2\sqrt{d}/r(i)\big)\leq\epsilon M\frac34 r(i)\|\v x-\v y\|_2.
\end{eqnarray*}
Also
$$
m_\ell=\|\ell(1)-\ell(0)\|_2=
\big\|\lfloor r(i)\v y\rfloor-\lfloor r(i)\v x\rfloor\big\|_2\geq
r(i)\|\v x-\v y\|_2-\sqrt{d}>\frac34 r(i)\|\v x-\v y\|_2
$$
and this finishes the proof of the second condition.

Finally, we prove the third condition, \ref{item:z-iii}.
Let $g_i:\v u\mapsto\lfloor r(i)\v u\rfloor+\v z_{r(i)}$.
For a set $U\subset g_i([-1,1]^d)$ we have
$$|U|=r(i)^d\lambda(g_i^{-1}(U))\quad\text{ and }\quad |A\cap U|=r(i)^d \int_{g_i^{-1}(U)}\phi_id\lambda.$$
We wish to apply these two facts to $U=\KZ=\KZ(\epsilon,\ell)$ (as in \DefinitionRef{XZ}).
The idea is to approximate $g_i^{-1}(\KZ)$ with $\KR(\epsilon,\v x,\v y)$.
Now let $\KZ=\KZ(\epsilon,\ell)$ and $\KR=\KR(\epsilon',\v x,\v y)$ be as in and \DefinitionRef{XR}.
For any $\v u\in\KR$ there is some $t\in[0,1]$ such that
$$\left\|\v u-(1-t)\v x-t\v y\right\|_2\leq\epsilon'\|\v x-\v y\|_2.$$
It follows from \ref{item:i2} that
$$\Big\|g_i(\v u)-\big\lfloor(1-t)g_i(\v x)+tg_i(\v y)\big\rfloor\Big\|_2\leq r(i)\epsilon'\|\v x-\v y\|_2+2\sqrt{d}\leq\epsilon m_\ell$$
This implies that $\KR\subset g_i^{-1}(\KZ)$.
Similarly, one can show that $g_i^{-1}\big(\KZ(\epsilon,\ell)\big)\subset\KR(4\epsilon',\v x,\v y)$.
Therefore we conclude that
$$\frac{|A\cap\KZ|}{|\KZ|}\geq
\frac{r(i)^d \int_\KR \phi_i d\lambda }{\lambda(\KR\big(4\epsilon',\v x,\v y)\big)r(i)^d}>
\frac{\constantone'}{4^d}=\constantone.$$
This finishes the proof.
\end{proof}

\section{Proof of \TheoremRef{A} using \LemmaRef{XZ}}\label{section:Conclusion}

Assume $\LipZ\neq \emptyset$ and let $f\in\LipZ$. Let $\delta>0$ and $\v w\in S^d$ be given by \LemmaRef{XZ}.
We will assume without loss of generality that the first coordinate $w_1$ of $\v w$ has the highest absolute value. Since $\|\v w\|_2=1$, this implies that
\begin{equation}
\label{equation:lowerboundw1}
|w_1|\geq d^{-1/2}.
\end{equation}
We will need the following form of Dirichlet's approximation theorem.
\begin{Lemma}\label{lemma:dirichlet}
Let $(u_2,u_3,\dots,u_{d+1})\in\R^d$ and let $N\in\N$.
Then there exists a positive integer $b\leq N^d$ and
$a_2,a_3,\dots,a_{d+1}\in\Z$ such that
$$\left|u_l-\frac{a_l}b\right|\leq\frac1{bN}\qquad\forall l\in\{2,3,\dots,d+1\}.$$
\end{Lemma}
For the remainder of this section let $N$ be any positive integer satisfying $N>kc_2/(\delta c_3)$, where $c_2$ is the constant appearing on Lemma \ref{lemma:Enorm} and $c_3$ is the constant appearing in Lemma 
\ref{lem:2ndrev}, both depending only on the fixed parameters $d$ and $M$. Also, assume that
$N$ is large enough such that $\XZ{f,A}{\epsilon}{\delta}{\v w}\neq\emptyset$ for all $\epsilon\leq\frac1N$,
as guaranteed by \LemmaRef{XZ}.
Apply
\LemmaRef{dirichlet} to find $b,a_2,\dots,a_{d+1}\in\Z$ satisfying
\begin{equation}
\label{equation:conclusion3}
\left|\frac{w_l}{w_1}-\frac{a_l}b\right|\leq\frac1{bN}\qquad\forall\,l\in\{2,3,\dots,d+1\}.
\end{equation}
Finally let $\epsilon=\frac1{bN}$ and take some generalized line
segment $\ell\in\XZ{f,A}{\epsilon}{\delta}{\v w}$.

We need a lower bound on the cardinality of $\KZ$.
\begin{Lemma}\label{lem:2ndrev}
There exists a constant $\constantseven$ that only depends on the dimension $d$ such that for any generalized line segment $\ell$ and any $\epsilon>0$ satisfying $\epsilon m_\ell>14\sqrt{d}$, the cylinder $\KZ$ has cardinality
$|\KZ|\geq \constantseven \epsilon^{d-1} m_\ell^d$.
\end{Lemma}
\begin{proof}
Define $\tilde{\KZ}:=\{\vec z \in\R^d: \lfloor{\vec z}\rfloor\in\KZ\}$. Observe that $\lambda(\tilde{\KZ})=|\KZ|$,
hence it suffices to show that $\lambda(\tilde{\KZ})\geq \constantseven \epsilon^{d-1} m_\ell^d$.
Let $$\gamma(t):=(1-t)\ell(0)+t\ell(1),\qquad\forall t\in[0,1].$$
If $\vec z\in\R^d$ satisfies $\min_{t\in[0,1]}\|\vec z - \gamma(t)\|\leq \epsilon m_\ell-2\sqrt{d}$,
then a simple application of the triangle inequality implies that
$\min_{t\in[0,1]}\|\lfloor\vec z\rfloor - \ell(t)\|\leq \epsilon m_\ell$. In other words this shows that $\tilde{\KZ}$
contains the cylinder with axes $\gamma(t)$ and radius $\epsilon m_\ell-2\sqrt{d}$. Since we assume that
$\epsilon m_\ell>14\sqrt{d}$ it is guaranteed that $\epsilon m_\ell-2\sqrt{d}>0$. This allows us to bound the size of
$\lambda(\tilde{\KZ})$ from below by a constant multiple of $(\epsilon m_\ell-2\sqrt{d})^{d-1}m_\ell\leq\big(\frac{\epsilon m_\ell}{2}\big)^{d-1}m_\ell$. Thus,
$$
\lambda(\tilde{\KZ})\geq \constantseven \epsilon^{d-1} m_\ell^d
$$
for some $\constantseven$ that only depends on $d$.
\end{proof}

The idea is to cover $f(\KZ)$ with at most $\delta |\KZ| /k$ lines, where
$\KZ=\KZ(\epsilon,\ell)$ is as in \DefinitionRef{XZ}.
Due to
\ref{item:z-iii} we have that $|\KZ\cap A|/|\KZ|>\delta$, therefore this will imply that there exists $X\subset\KZ\cap A$ with $|X|=k$ and such that the image $f(X)$ is contained in a line.

We can assume without loss of generality that
$\fcircell(0)=0$, as otherwise we can instead cover
the set $f(\KZ)-\fcircell(0)$ with less than $\delta |\KZ| /k$ lines and
this would then yield a covering of $f(\KZ)$ with the same number of lines.

Define $\v s=(b,a_2,a_3,\ldots,a_{d+1})\in \Z^{d+1}$ and let $E$ be the set of all lines in $\R^{d+1}$ of the form
$\{\v x-t \v s:t\in\R\}$ for some $\v x\in f(\KZ)$.
Thus the set $E$ covers all points in $f(\KZ)$.

\begin{Lemma}
\label{lemma:Enorm}
There exists a constant $c_2$, depending only on the dimension $d$ and on the Lipschitz constant $M$, such that
$$|E|\leq c_2\frac{m_\ell^d}{b^{d-1}N^d}.$$
\end{Lemma}

Before we embark on the proof of this lemma let us first show how it implies \TheoremRef{A}:
On the one hand, it follows from Lemma \ref{lem:2ndrev}, condition \ref{item:z-iii} and the choice of $\epsilon$ that $|A\cap\KZ|>\delta\constantseven m_\ell^d/(bN)^{d-1}$.
On the other hand, \LemmaRef{Enorm} tells us that we can cover the image of $\KZ$ under $f$ with no more than $c_2\big(m_\ell\big)^d/(b^{d-1}N^d)$ lines.
It follows from the pigeonhole principle that some line in $E$ contains the image, under $f$, of at least

$$\frac{|A\cap\KZ|}{|E|}\geq\constantone
\left(\constantseven\frac{m_\ell^d}{b^{d-1}N^{d-1}}\right)\Big/
 \left(c_2\frac{m_\ell^d}{b^{d-1}N^d}\right)=
 \frac{\constantone \constantseven}{c_2}N$$
 points from $A$.
By choosing $N$ sufficiently large, depending only on $d,k,M$ and $\delta$, we deduce that some
line in $E$ must contain the image of at least $k$ points from $A$.
This contradicts the fact that $f\in\LipZ$, and this contradiction finishes the proof of \TheoremRef{A}.

Now, all that remains to show is \LemmaRef{Enorm}.
Since all lines in $E$ are parallel,
in order to count them, we can simply look at their intersection
with the hyperplane $H=\{0\}\times\R^d$.

With this in mind, for a vector $\v u=(u_1,\dots,u_{d+1})$ with $u_1\neq0$,
we define the projection $P_{\v u}:\R^{d+1}\to\R^d$ by
$$P_{\v u}(x_1,\dots,x_{d+1})=(x_2,\dots,x_{d+1})-
\tfrac{x_1}{u_1}(u_2,\dots,u_{d+1}).$$
Note that $\big(0,P_{\v u}(\v x)\big)\in\R^{d+1}$ is the intersection
of the line $\{\v x-t \v u:t\in\R\}$ with $H$.
Thus $E_0:=P_{\v s}(f(\KZ))$ is the set of intersections of lines in
$E$ with $H$, and hence $|E_0|=|E|$.

A simple calculation shows that if $\v x,\v y\in \Z^{d+1}$
are such that the first coordinate of $\v x$ and the first
coordinate of $\v y$ differ by a multiple of $b$, then
$P_{\v s}(\v x)-P_{\v s}(\v y)\in\Z^d$. This implies
\begin{equation}
\label{equation:E0lattice}
E_0\subset \bigcup_{l\in[0,b-1]}\left(\v P_{\v s}(l,0,\dots,0)+\Z^d\right).
\end{equation}

Next we want to enclose $E_0=P_{\v s}(f(\KZ))$ inside a convex
set $D\subset\R^d$.
It follows from \eqref{equation:lowerboundw1} that the operator norm of $P_{\v s}$ is smaller than a constant $\constanteight$ which only depends on $d$.
Let $\constantfunf=2\constanteight M$,  let $\v u=P_{\v s}\big(\fcircell(1)\big)\in\R^d$ and define
$$D=\{t\v u:t\in[0,1]\}+B_\R\big(0,\constantfunf\epsilon m_\ell\big).$$

To see that $E_0\subset D$, let $\v x\in E_0$ be arbitrary.
From the above construction we have $\v x=P_{\v s}\big(f(\v z)\big)$ for some $\v z\in\KZ$.
Therefore there exists some $t\in[0,1]$ such that $\|\v z-\ell(t)\|_2\leq\epsilon m_\ell$ and hence $\|f(\v z)-\fcircell(t)\|_2\leq\epsilon Mm_\ell$.
Using \ref{item:z-ii} we deduce that $\|f(\v z)-t\fcircell(1)\|_2\leq 2\epsilon Mm_\ell$.
Thus
$$\|\v x-t\v u\|_2=\Big\|P_{\v s}\big(f(\v z)\big)-P_{\v s}\big(t\fcircell(1)\big)\Big\|_2\leq2\constanteight\epsilon Mm_\ell.$$
This shows $\v x\in D$ as desired.

Putting together the inclusion $E_0\subset D$ with \EquationRef{E0lattice} we deduce that
\begin{equation}\label{equation:e0D}|E_0|\leq c_6b\lambda(D)
\end{equation} for some constant $c_6$ that only depends on $d$.
Moreover $\lambda(D)$ can be bounded by
\begin{equation}\label{equation:lambdaestimate}\lambda(D)\leq\constantnine\left(\constantfunf\epsilon m_\ell\right)^{d-1}\cdot\left(\|u\|_2+2\constantfunf\epsilon m_\ell\right)\end{equation}
where $\constantnine$ is the volume of the unit ball in $\R^{d-1}$.
Finally we need to estimate $\|\v u\|_2$.

Let
$$
\v v=\big(v_1,\dots,v_{d+1}\big)= \frac{\fcircell(1)-
\fcircell(0)}{\big\|\fcircell(1)-\fcircell(0)\big\|_2}
= \frac{\fcircell(1)}{\big\|\fcircell(1)\big\|_2}
$$
and define
$\tilde s=(a_2,\dots,a_{d+1})$ and $\tilde v=\big(v_2,\dots,v_{d+1}\big)$.
We claim that there exists a constant $\constantsieben$ which only depends on $d$ such that for all $\v x\in\R^{d+1}$ we have
\begin{equation}\label{equation:Pnorm}\left\|P_{\v s}(\v x)-P_{\v v}(\v x)\right\|_2\leq \constantsieben\epsilon\|\v x\|_2.
\end{equation}
To prove this claim, first observe that for $\v x\in\R^{d+1}$ we have
\begin{equation}
\label{equation:conclusion4}
\left\|P_{\v s}(\v x)-P_{\v v}(\v x)\right\|_2=
\left\|\frac{x_1}{v_1}\tilde v-\frac{x_1}b\tilde s\right\|_2
\leq\|\v x\|_2\left\|\frac{\tilde v}{v_1}-
\frac{\tilde s}b\right\|_2.
\end{equation}
Next, take an arbitrary $i\in\{2,3,\dots,d+1\}$; it follows from \EquationRef{conclusion3} that
\begin{equation}\label{equation:conclusionclaimbound}\left|\frac{v_i}{v_1}-\frac{a_i}b\right|\leq\epsilon+\left|\frac{v_i}{v_1}-\frac{w_i}{w_1}\right|.\end{equation}
From \ref{item:z-i} we get that $\|\v w-\v v\|_2\leq\epsilon$, and so, in particular, $|w_j-v_j|<\epsilon$ for each $j\in\{1,\dots,d+1\}$.
Recall that $\epsilon<d^{-1/2}/2$, $w_1\geq d^{-1/2}$ and $|w_1|,|w_i|\leq1$.
We deduce that
$$\left|\frac{v_i}{v_1}-\frac{w_i}{w_1}\right|=\left|\frac{v_iw_1-v_1w_i}{w_1v_1}\right|=\left|\frac{(v_i-w_i)w_1-(v_1-w_1)w_i}{w_1v_1}\right|\leq\frac{2\epsilon}{1/(2d)}=4d\epsilon.$$
Putting this together with \EquationRef{conclusionclaimbound} and \EquationRef{conclusion4} we get \EquationRef{Pnorm} and this proves the claim.

Using \EquationRef{Pnorm} with $\v x=\fcircell(1)$ and observing that $P_{\v v}(\fcircell(1))=0$ we deduce that
$$\|\v u\|_2=\|P_{\v s}\fcircell(1)\|_2\leq\constantsieben\epsilon\|\fcircell(1)\|_2\leq\constantsieben\epsilon Mm_\ell.$$
Putting this together with \EquationRef{lambdaestimate} and \EquationRef{e0D} we conclude that
$$|E|=|E_0|\leq c_6b\lambda(D)\leq c_2\frac{m_\ell^d}{b^{d-1}N^d}$$
with $c_2=c_6c_7\constantfunf^{d-1}(c_8M+2\constantfunf)$.
This finishes the proof of \LemmaRef{Enorm} and hence the proof of \TheoremRef{A}.

\bibliography{refs-joel}

\begin{thebibliography}{1}

\bibitem{Dekking79}
F.~M. Dekking.
\newblock Strongly nonrepetitive sequences and progression-free sets.
\newblock {\em J. Combin. Theory Ser. A}, 27(2):181--185, 1979.

\bibitem{Gerver_Ramsey79}
J.~L. Gerver and L.~T. Ramsey.
\newblock On certain sequences of lattice points.
\newblock {\em Pacific J. Math.}, 83(2):357--363, 1979.

\bibitem{Heinonen05}
J.~Heinonen.
\newblock {\em Lectures on {L}ipschitz analysis}, volume 100 of {\em Report.
  University of Jyv\"askyl\"a Department of Mathematics and Statistics}.
\newblock University of Jyv\"askyl\"a, Jyv\"askyl\"a, 2005.

\bibitem{Justin72}
J.~Justin.
\newblock Characterization of the repetitive commutative semigroups.
\newblock {\em J. Algebra}, 21:87--90, 1972.

\bibitem{Pomerance80}
C.~Pomerance.
\newblock Collinear subsets of lattice point sequences---an analog of
  {S}zemer\'edi's theorem.
\newblock {\em J. Combin. Theory Ser. A}, 28(2):140--149, 1980.

\bibitem{Ramsey77}
L.~T. Ramsey.
\newblock Fourier-{S}tieltjes transforms of measures with a certain continuity
  property.
\newblock {\em J. Functional Analysis}, 25(3):306--316, 1977.

\bibitem{Szemeredi75}
E.~Szemer{\'e}di.
\newblock On the sets of integers containing no $k$ elements in arithmetic
  progressions.
\newblock {\em Acta Arith.}, 27:299--345, 1975.

\bibitem{vdWaerden27}
B.L. van~der Waerden.
\newblock Beweis einer baudetschen vermutung.
\newblock {\em Nieuw. Arch. Wisk.}, 15:212--216, 1927.

\end{thebibliography}
\bibliographystyle{plain}
\end{document}